\documentclass[a4,12pt,reqno]{amsart}
\usepackage{amsthm}
\usepackage{amsmath}
\usepackage{amssymb}
\usepackage{latexsym}
\usepackage{enumerate}
 \newtheorem{theorem}{Theorem}[section]
 \newtheorem{lem}[theorem]{Lemma}
 \newtheorem{prop}[theorem]{Proposition}
 \newtheorem{cor}[theorem]{Corollary}
\theoremstyle{remark} 
 \newtheorem{rem}{Remark}[section]
  
\makeatletter
\@addtoreset{equation}{section}

\makeatother

\newcommand{\N}{\mathbb{N}}
\newcommand{\R}{\mathbb{R}}
\newcommand{\C}{\mathbb{C}}
\newcommand{\Nz}{\mathbb{N}_0}

\newcommand{\Sc}{\mathcal{S}}

\newcommand{\id}{\mathrm{id}\,}
\newcommand{\F}{\mathcal{F}}
\newcommand{\Pt}{\partial_t}
\newcommand{\Px}{\partial_x}
\newcommand{\Pj}{\partial_j}
\newcommand{\Pxi}{\partial_\xi}
\newcommand{\Rep}{\mathrm{Re}\,}
\newcommand{\Imp}{\mathrm{Im}\,}
\newcommand{\vD}{\varDelta}
\newcommand{\tr}[1]{{}^t\hspace{-1mm}#1}
\newcommand{\diag}{\mathrm{diag}}

\newcommand{\bra}[1]{\left< #1 \right>}
\newcommand{\braa}[1]{\left( #1 \right)}
\newcommand{\brab}[1]{\left\{ #1 \right\}}
\newcommand{\brac}[1]{\left[ #1 \right]}
\newcommand{\ab}[1]{\left| #1 \right|}
\newcommand{\nr}[1]{\left\| #1 \right\|}

\newcommand{\mf}{\mathfrak}

\newcommand{\ep}{\epsilon}
\newcommand{\vep}{\varepsilon}

\newcommand{\vc}[3]{
\left(
  \begin{array}{c}
#1\\
#2\\
#3\\
  \end{array}
\right)
}

\newcommand{\vcc}[3]{
\left( #1,#2,#3 \right)}

\newcommand{\mtrx}[9]{
\left(
  \begin{array}{ccc}
#1&#2&#3\\
#4&#5&#6\\
#7&#8&#9\\
  \end{array}
\right)
}

\begin{document}
\title[Inverse Scattering Problems]{
Inverse scattering problems for the Hartree equation 
whose interaction potential decays rapidly.}
\author[H. Sasaki]{
Hironobu Sasaki$^*$ \\
Department of Mathematics and Informatics, 
Chiba University, 
263--8522, 
Japan.
}
\thanks{$^*$Supported by Grant-in-Aid 
for Young Scientists (B) 
No. 22740082.}
\thanks{$^*$email:
\texttt{sasaki@math.s.chiba-u.ac.jp}} 
\subjclass[2000]{35Q55, 81U40, 35P25}
\keywords{
Schr\"odinger equation; 
inverse scattering; 
interaction potential
} 
\maketitle

\begin{abstract}
We consider 
inverse scattering problems for 
the three-dimensional Hartree equation.
We prove that 
if the unknown interaction potential 
$V(x)$ of the equation satisfies 
some rapid decay condition,
then we can uniquely determine 
the exact value of 
$\Pxi^\alpha \widehat{V}(0)$ 
for any multi-index $\alpha$
by the knowledge of the scattering operator 
for the equation.
Furthermore,
we show some stability estimate for identifying 
$\Pxi^\alpha \widehat{V}(0)$.
\end{abstract}

\section{Introduction}\label{sec:intro}
This paper is concerned with 
inverse scattering problems for 
the three-dimensional Hartree equation
\begin{align}\label{eq:H}
i\Pt u + \Delta u = ( V \ast |u|^2 )u,
\quad
(t,x)\in \R\times \R^3.
\end{align} 
Here,
$u=u(t,x)$
is a complex-valued unknown function,
$i=\sqrt{-1}$,
$\Pt =\partial/\partial t$,
$\Delta = \partial_1^2 + \partial_2^2 + \partial_3^2$,
$\Pj =\partial/\partial x_j$ ($j=1,2,3$),
$V=V(x)$ is a real-valued measurable function,
the symbol $\ast$ denotes the convolution in $\R_x^3$. 
The equation (\ref{eq:H})
is approximately derived by 
the time-dependent multi-body Schr\"odinger equation
and the function $V(x)$ means the interaction potential 
for particles. \\

Our aim of this paper is 
to identify the unknown interaction potential $V(x)$
by the knowledge of the scattering operator.
In order to mention the existence of 
the scattering operator for the equation (\ref{eq:H}),
we introduce some notation.  
For $1\le p \le \infty$,
we denote the Lebesgue space $L^p(\R^3)$ and its norm 
by $L^p$ and $\| \cdot \|_p$,
respectively.  
For $1\le p \le \infty$,
we denote the usual Sobolev space
$W^1_p(\R^3)$ by $W^1_{p}$ 
and we put $H^1=W^1_2$.
For a Banach space $Y$ and for $1\le q \le \infty$, 
we put $C_b Y = C(\R;Y) \cap L^\infty(\R;Y)$ 
and 
$L^q Y = L^q(\R;Y)$.
For a Banach space $Y$ equipped with the norm $\| \cdot \|_Y$
and for $\delta>0$,
we define the closed ball $B(\delta,Y)$ by 
$B(\delta,Y) = \{ f\in Y;\ \| f \|_Y \le \delta \}$.
For $t\in \R$, 
the linear operator 
$U(t)$ is defined by 
$U(t)=e^{it\Delta}$.
Then $v(t)=U(t)\phi$ ($\phi\in L^2$) 
solves the free Schr\"odinger equation 
$i\Pt v + \Delta v = 0$
with the initial condition $v(0)=\phi$.  
From Strauss \cite{Strauss1981},
we see the existence of 
the scattering operator for (\ref{eq:H}):

\begin{theorem}\label{thm:direct}
Let $M>0$.
Assume that $V \in B(M,L^1)$.
Then there exist some positive numbers
$\delta_0(M)$ and $C_0$ 
depending only on $M$
such that the following properties hold:
\begin{enumerate}[(1)]
  \item 
For any $\phi_-\in B(\delta_0(M),H^1)$,
there uniquely exists 
a time-global solution
$u\in Z:=C_b H^1 \cap L^3 W^1_{18/5}$
to (\ref{eq:H}) 
such that  
$u(0)\in B(C_0\delta_0(M),H^1)$
and
\begin{align*}
\lim_{t\to -\infty} 
\nr{ u(t) - U(t)\phi_- }_{H^1}
=0.
\end{align*}  
Thus, 
the wave operator
\[
\mathcal{W}_-: 
B(\delta_0(M);H^1) 
\ni \phi_- \mapsto u(0) \in
B(C_0\delta_0(M);H^1)
\]
is well-defined. 
  \item 
For any 
$\phi_0\in B(C_0\delta_0(M);H^1)$,
there uniquely exist 
$v\in Z$ 
and 
a datum 
$\phi_+ \in B(C_0^2\delta_0(M);H^1)$
such that 
$v$ solves 
(\ref{eq:H})
with the initial condition 
$v(0)=\phi_0$
and 
satisfies  
\begin{align*}
\lim_{t\to +\infty} 
\nr{ v(t) - U(t)\phi_+ }_{H^1}
=0 .
\end{align*} 
Thus, 
the inverse wave operator 
\[
\mathcal{V}_+: 
B(C_0\delta_0(M);H^1)
\ni \phi_0 \mapsto \phi_+ \in
B(C_0^2\delta_0(M);H^1)
\]
and the scattering operator 
\[
S=\mathcal{V}_+ \circ \mathcal{W}_-: 
B(\delta_0(M);H^1) \to
B(C_0^2\delta_0(M);H^1)
\]
are well-defined. 
  \item 
We have
for any 
$\phi_- \in B(\delta_0(M),H^1)$,
\begin{align}\label{eqiv:S op}
S(\phi_-)
=
\phi_-
-
i\int_{\R}
U(-t) 
( V\ast |u(t)|^2 ) u(t) 
dt
\end{align} 
and 
\begin{align}\label{ineq:u - U phi}
\nr{ u(t) - U(t)\phi_- }_Z
\le C_0 
\nr{ \phi_- }_{H^1}^3,
\end{align} 
where $u(t)$ is the time-global solution 
mentioned in (1).
\end{enumerate}
\end{theorem}

\begin{rem}\label{rem:direct}
For the equation (\ref{eq:H}),
the scattering operator is well-defined 
on some neighborhood of 0 
in some suitable Hilbert space 
if the interaction potential $V$ satisfies either 
$
\sup_{x\in \R^3}
|x|^\gamma
\ab{V(x)} 
<\infty
$
for some $\gamma\in (1,3)$ or 
$V\in L^r$
for some $r\in [1,3)$.
For a proof, 
see
\cite{HayashiTsutsumi,Mochizuki,Strauss1981}.
\end{rem}

The inverse scattering problem 
for the perturbed Schr\"odinger equation  
is to recover the perturbed term   
by applying the knowledge of 
scattering states.
Before we introduce our main results,
we first review 
some known results of 
inverse scattering problems for 
$n$-dimensional 
nonlinear Schr\"odinger equations
briefly.
Strauss \cite{Strauss1974}
considered the nonlinear
Schr\"odinger equation
\begin{align*}
i\Pt u + \Delta u
=
W|u|^{p-1}u,
\quad
(t,x)\in \R \times \R^n.
\end{align*} 
Here, 
$p$ is a given number and 
$W=W(x)$ is unknown.
It was proved that 
if $p$ and $W$ satisfy
suitable conditions,
then the unknown $W$ is uniquely reconstructed by 
\begin{align*}
W(x_0)
=
\frac{\displaystyle{
\lim_{R\to \infty}R^{\,n+2}
K_p  \brac{\phi_{R,x_0}}
}}{\displaystyle{
\int_{\R\times \R^n}
\ab{U(t)\phi(x)}^{p+1}
d(t,x)
}}
,
\quad
x_0 \in \R^n,\
\phi\in H^1(\R^n) \cap L^{1+1/p}(\R^n)\setminus \{ 0 \},
\end{align*} 
where 
$\phi_{R,x_0}(x)=\phi(R(x-x_0))$
($R>0$, $x\in \R^n$)
and
\begin{align}\label{SAL}
K_p[\phi]
=
\lim_{\vep\to 0}
i\vep^{-p}
\bra{
(S-\id)(\vep\phi),\phi
}_{L^2(\R^n)},
\end{align} 
which is called 
the small amplitude limit,
the notation 
``$\id$" is the identity mapping.
Later, Weder 
\cite{Weder1997,Weder2000,
Weder2001-1,Weder2001-2,Weder2005-1,Weder2005-2}
proved that a more general class of nonlinearities 
is uniquely reconstructed, 
and moreover, 
a method is given for the unique reconstruction 
of the potential that acts 
as a linear operator.

Unfortunately,    
the above methods to obtain the reconstruction formula 
are not applicable to 
the case of the Hartree term 
$(V\ast |u|^2) u$
(for details, see, e.g. Section 1 of \cite{Sasaki2008}). 
Watanabe \cite{Watanabe2001-1,Watanabe2001-2,
Watanabe2002,Watanabe2007-1} 
studied 
the Hartree equation with a potential  
\begin{align*}
i\Pt u + \Delta u 
=
W_1 u + (W_2 \ast |u|^2)u,
\quad
(t,x)\in \R \times \R^n.
\end{align*} 
Here,  
$W_1(x)$ satisfies some suitable condition,  
$W_2(x)= Q|x|^{-\gamma}$ 
($Q>0$ and 
$\gamma\in [2,4]$ with $\gamma<n$).
It was proved that $W_1$, $Q$ and $\gamma$ 
are uniquely determined
(see also 
\cite{Watanabe2007-2,Sasaki2007-2}).
Sasaki \cite{Sasaki2008} considered 
the three-dimensional Hartree equation with a potential  
\begin{align*}
i\Pt u + \Delta u 
=
W_3 u
+
(W_4 \ast |u|^2)u,
\quad
(t,x)\in \R \times \R^3,
\end{align*} 
where 
$W_j = Q_j \exp(-\gamma_j|x|)/|x|$ 
($Q_j\in \R$ and $\gamma_j>0$, $j=3,4$).
It was shown that if $|Q_3|<\gamma_3$, 
then
unknown parameters $Q_j$ and $\gamma_j$ ($j=3,4$)
are uniquely determined. 
In the above inverse scattering problems
for the Hartree equation,
we assume that 
the interaction potential 
is a known function with unknown parameters.
In other words,
we already know 
what kind of shape 
the interaction potential has.

We next consider 
inverse scattering problems for the Hartree equation 
and 
the case where we do not know 
what kind of shape 
the interaction potential $V$ has.
Sasaki--Watanabe \cite{SasakiWatabnabe} 
proved that 
if $2\le n \le 6$, $V\in L^1(\R^n)$ and 
$\phi\in H^1(\R^n) \setminus \{ 0 \}$,
then 
the following formula holds:
\begin{align*}
\F V(0)
=
\frac{\displaystyle{
\lim_{\lambda\to 0}
\lambda^{n+2}
K_3\brac{\phi_{\lambda,0}}
}}{\displaystyle{
(2\pi)^{n/2}
\nr{U(t)\phi }_{L^4(\R;L^4)}^4
}},
\end{align*} 
where 
$\F V$ is the Fourier transform of $V$ 
and 
$K_3[\phi]$ is 
the small amplitude limit defined by  
(\ref{SAL}) with $p=3$.
Later, 
Sasaki \cite{Sasaki2007-1}
proved that 
if $n\ge 3$ and $V$ is radial and 
satisfies some decay condition,
then we have for any non-negative integer $m$,
\begin{align*}
\nu_m
:=
\left.
\dfrac{d^m}{d\rho^m}
\F V(\rho,0,\cdots,0)
\right|_{\rho=0}
=
\frac{\displaystyle{
\lim_{\lambda \to 0}
\dfrac{\partial^m}{\partial \lambda^m}
\braa{
\lambda^{n+2}
K_3\brac{\phi_{\lambda,0}}}
}}{\displaystyle{
\int_{\R}
\nr{
|\cdot|^{m/2}
\F \ab{
U(t)\phi
}^2
}_{L^2(\R^n)}^2
dt
}},
\quad
\phi\in \Sc(\R^n)\setminus \{ 0 \}.
\end{align*} 
Furthermore,
if we assume that $\F V$ is an entire function,
then we have 
\begin{align}\label{radial}
\F V(\xi)
=
\sum_{m=0}^\infty
\dfrac{\nu_m}{m!}
|\xi|^m,
\quad
\xi\in \R^n
\end{align}  
and we can hence reconstruct $V$.

In this paper, 
we consider inverse scattering problems for (\ref{eq:H}),
supposing some decay condition for $V$.
We show that 
for any multi-index $\alpha$,
we can uniquely determine the exact value of 
$\Pxi^\alpha \F V(0)$ 
and 
we can reconstruct $V$ 
even if we do not suppose 
any symmetric conditions for $V$.

\subsection{Notation}
To state our main results precisely,
we introduce some notation.
Let $\N$ be the set of all positive integers.
Put $\Nz = \N \cup \{ 0 \}$.
We denote the Schwartz class $\Sc(\R^3)$ by $\Sc$.  
The Fourier transform $\F$ on $L^1$ 
is defined by 
\begin{align*}
\F f(\xi) = 
(2\pi)^{-3/2}
\int_{\R^3} 
e^{-ix\cdot \xi} f(x)
dx,
\quad 
f\in L^1, 
\ 
\xi\in \R^3. 
\end{align*} 
We define $3\times 3$ matrices $I_m$ ($m=1,2,3$) by 
\begin{align*}
I_1 = \mtrx{1}{0}{0}{0}{1}{0}{0}{0}{1},
\quad
I_2 = \mtrx{0}{1}{0}{1}{0}{0}{0}{0}{1},
\quad
I_3 = \mtrx{0}{0}{1}{0}{1}{0}{1}{0}{0}.
\end{align*} 
For multi-index 
$\ep = (\ep_1,\ep_2,\ep_3)\in \{ 0,1 \}^3$,
we put  
\begin{align*}
m(\ep)
=
\left\{
  \begin{array}{ll}
1                      &\text{if $|\ep| =    0$,}\\
\min\{ m;\ \ep_m=1 \}  &\text{if $|\ep| \neq 0$ }\\
  \end{array}
\right.
\end{align*} 
and $I(\ep)=I_{m(\ep)}$.
For $N\in \Nz$ and $\ep\in \{ 0,1 \}^3$,
we define 
$3\times 3$ matrices 
$D^{\lambda,\mu}_{N,\ep}$
($\lambda,\mu>0$),
non-negative integers 
$P_{N,\ep}(\alpha)$ 
($\alpha =(\alpha_1,\alpha_2,\alpha_3) \in \Nz^3$),
linear operators 
$U^\mu_{N,\ep}(t)$ and 
$U_\ep(t)$
($\mu>0$, $t\in \R$)
by 
\begin{align*} 
D^{\lambda,\mu}_{N,\ep}
&=
\lambda I(\ep)
\mtrx{1}{0}{0}{0}{\mu}{0}{0}{0}{\mu^{N+1}}
I(\ep),
\quad
P_{N,\ep}(\alpha)
=
\vcc{0}{1}{N+1}
I(\ep) \vc{\alpha_1}{\alpha_2}{\alpha_3} ,\\
U^\mu_{N,\ep}(t)
&=
\F^{-1}
\exp\braa{-it\ab{ D^{1,\mu}_{N,\ep}\xi }^2}
\F ,
\quad
U_\ep(t)
=
\F^{-1}
\exp\braa{-it \xi_{m(\ep)}^2}
\F ,
\end{align*} 
respectively.
We now enumerate two examples. 
If $\ep =(1,0,1)$, then 
\begin{align*}
I(\ep)
&=I_1,
\quad
D^{\lambda,\mu}_{N,\ep}
=
\diag \braa{\lambda,\lambda \mu ,\lambda \mu^{N+1}},
\quad
P_{N,\ep}(\alpha)
=
\alpha_2 + (N+1)\alpha_3,
\\
U^\mu_{N,\ep}(t)
&=
\exp
\braa{
it\brab{
\partial_1^2 + 
\mu^2 \partial_2^2 + 
\mu^{2N+2}\partial_3^2  
}},
\quad
U_\ep(t)
=
\exp\braa{it\partial_1^2}.
\end{align*} 
If $\ep =(0,0,1)$, then 
\begin{align*}
I(\ep)
&=I_3,
\quad
D^{\lambda,\mu}_{N,\ep}
=
\diag \braa{\lambda \mu^{N+1},\lambda \mu ,\lambda},
\quad
P_{N,\ep}(\alpha)
=
\alpha_2 + (N+1)\alpha_1,
\\
U^\mu_{N,\ep}(t)
&=
\exp
\braa{
it\brab{
\mu^{2N+2}\partial_1^2  + 
\mu^2 \partial_2^2 + 
\partial_3^2  
}},
\quad
U_\ep(t)
=
\exp\braa{it\partial_3^2}.
\end{align*} 
For $N\in \Nz$, 
we put 
$
N^\ast = 
\# \left\{ 
\alpha\in \Nz^3;\ 
|\alpha| = N
\right\}
=
(N+1)(N+2)/2
$.
For any $N\in \Nz$ and $\ep \in \left\{ 0,1 \right\}^3$,
there uniquely exists 
a sequence of multi-indices 
$
\brab{ \alpha(k) }_{k=1}^{N^\ast}
\subset
\Nz^3
$
satisfying the two following properties:
\begin{itemize}
  \item 
It follows that 
$|\alpha(k)|=N$
for any $k=1,\cdots,N^\ast$.
  \item 
If $1\le k < l \le N^\ast$,
then 
$P_{N,\ep}(\alpha(k))<P_{N,\ep}(\alpha(l))$.
\end{itemize}
Let us denote such a sequence by 
$
\brab{ \alpha(N,\ep;k) }_{k=1}^{N^\ast}
$.
For a $3 \times 3$ matrix $D$ and 
for a function $\varphi:\R^3 \to \C$,
we put $\varphi\circ D(x)=\varphi (Dx)$ 
($x\in \R^3$).
For $L\in \Nz$, $\lambda>0$ and 
for a function $h:(0,\infty) \to \C$,
we define 
$\vD_\lambda^L h(\lambda)$ 
by
\begin{align*}
\vD_\lambda^L h(\lambda)
=
\sum_{l=0}^L
\frac{(-1)^{L-l}L!}{l!(L-l)!} 
h( (l+1)\lambda ) .
\end{align*}  
For 
$\phi,\widetilde\phi \in \Sc$,
$\mu,\lambda >0$,
$N\in \Nz$ 
and 
$\ep \in \{ 0,1 \}^3$,
we set 
\begin{align*}
&
J^\lambda_{N,\ep}
\brac{ \phi,\widetilde\phi,\mu }
= \\
&\quad
\dfrac{i\mu^{N+2}\lambda^{-2N-|\ep|}}{
(2\pi)^{3/2}}
\vD_\lambda^{2N+|\ep|}
\brab{
\lambda^{-3N-7}
\bra{
(S-\id)
\left(
\lambda^{N+4}\phi\circ D^{\lambda,\mu}_{N,\ep}
\right),
\widetilde\phi \circ D^{\lambda,\mu}_{N,\ep}
}}  ,
\\
&
\Phi_{N,\ep}
\braa{ \phi,\widetilde\phi,\mu;t,\xi}
=
\F \ab{ U^\mu_{N,\ep}(t)\phi }^2 (\xi)
\times
\overline{
\F\braa{
\overline{U^\mu_{N,\ep}(t)\phi}
\times
U^\mu_{N,\ep}(t) \widetilde\phi 
}
(\xi)
}         ,\\
&
\Phi_\ep
\braa{ \phi,\widetilde\phi;t,\xi }
=
\F \ab{ U_\ep(t)\phi  }^2 (\xi)
\times
\overline{
\F\braa{
\overline{U_\ep(t)\phi}
\times
U_\ep(t) \widetilde\phi 
}
(\xi)
} .
\end{align*} 
Remark that 
if we fix $N$, $\ep$ and $\mu$, 
then
$\lambda^{N+4}\phi\circ D^{\lambda,\mu}_{N,\ep}$
belongs to the domain of the scattering operator $S$ 
for sufficiently small $\lambda>0$.
Furthermore,
we define 
$N^\ast$-th column vectors 
$\mf{J}^\lambda_{N,\ep}[ \phi,\widetilde\phi,\mu ]$,
$\mf{a}_{N,\ep}$
and 
the $N^\ast \times N^\ast$ matrix 
$\mf{M}_{N,\ep}[ \phi,\widetilde\phi,\mu ]$ by
\begin{align*} 
\mf{J}^\lambda_{N,\ep} 
\brac{ \phi,\widetilde\phi,\mu }
&=
\tr{
\braa{
J^\lambda_{N,\ep}
\brac{ \phi,\widetilde\phi,\mu^{(N+1)^{2(j-1)}}}
}_{1\le j \le N^\ast},          \\
\mf{a}_{N,\ep}}
&=
\tr{
\braa{
\partial_\xi^{2\alpha(N,\ep;k)+\ep}\F V(0)
}_{1\le k \le N^\ast}},          \\
\mf{M}_{N,\ep}
\brac{\phi,\widetilde\phi,\mu }
&=
\Bigg(
\dfrac{(2N+|\ep|)!}{(2\alpha(N,\ep;k)+\ep)!}
\mu^{ 
P_{N,\ep}\left( 2\alpha(N,\ep;k)+\ep \right)
\times (N+1)^{2(j-1)} 
}  
\\
&\qquad \times
\int_{\R^{1+3}}
\xi^{2\alpha(N,\ep;k)+\ep}
\Phi_{N,\ep}
\braa{ 
\phi,\widetilde\phi,\mu^{(N+1)^{2(j-1)}};t,\xi
}
d(t,\xi)
\Bigg)_{1\le j,k \le N^\ast} .
\end{align*} 
We remark that 
$\mf{J}^\mu_{N,\ep}$ and $\mf{M}_{N,\ep}$ are given 
and that 
$\mf{a}_{N,\ep}$ is unknown.
Let 
$E_1= \diag(-1,1,1)$,
$E_2= \diag(1,-1,1)$
and
$E_3= \diag(1,1,-1)$.
For $N\in \Nz$ and $\ep\in \{ 0,1 \}^3$, 
we set the three following conditions
for $\phi,\widetilde\phi\in \Sc$:
\begin{enumerate}[(C1)]
  \item 
For any $x \in \R^3$,
\begin{align*}
\phi \circ E_j (x) = \phi(x),
\quad 
j=1,2,3.
\end{align*} 
  \item 
For any $x\in \R^3$,
\begin{align*}
\widetilde\phi \circ E_j (x)
=
(-1)^{\ep_j} \widetilde\phi(x),
\quad 
j=1,2,3.
\end{align*} 
  \item 
For any $k\in \{ 1,\cdots,N^\ast \}$,
\begin{align*}
\int_{\R^{1+3}}
\xi^{2\alpha(N,\ep;k)+\ep}
\Phi_\ep\braa{\phi,\widetilde\phi;t,\xi}
d(t,\xi)
\neq 0.
\end{align*} 
\end{enumerate}

\subsection{Main Results}
We are ready to state our main results.
\begin{theorem}\label{thm:main-1}
Assume that 
the interaction potential $V$ 
of the equation (\ref{eq:H})
satisfies the following condition:
\begin{enumerate}[(V1)]
  \item 
There exists some positive number $A$ such that 
$e^{A|x|}V(x) \in L^1(\R_x^3)$.
\end{enumerate}
Let $N\in \Nz$ and $\ep\in \{ 0,1 \}^3$.
Fix 
$\phi,\widetilde\phi\in \Sc \setminus\{ 0 \}$. 
We assume that 
$\phi = \widetilde\phi$ and 
(C1) holds if $|\ep|=0$,  
and assume that 
(C1)--(C3) hold if $|\ep|\neq 0$. 

Then there exists some positive number $\widetilde\mu$ 
such that 
the matrix 
$
\mf{M}_{N,\ep}[ \phi,\widetilde\phi,\mu]
$
is invertible
for any $\mu\in (0,\widetilde\mu)$.
Furthermore,
for any $\mu\in (0,\widetilde\mu)$,
there exist some positive numbers $\lambda_\mu$ 
and 
$C_\mu$ 
depending only on 
$N,A,\| V \|_1,\mu,\phi$ and $\widetilde\phi$
such that 
\begin{align}\label{ineq:main}
\ab{
\mf{a}_{N,\ep}
-
\mf{M}_{N,\ep}
\brac{ \phi,\widetilde\phi,\mu }^{-1}
\mf{J}^\lambda_{N,\ep} 
\brac{ \phi,\widetilde\phi,\mu }
}
\le 
C_\mu\lambda,
\quad
\lambda\in (0,\lambda_\mu).
\end{align} 
In particular,
the unknown vector $\mf{a}_{N,\ep}$
is determined by 
\begin{align}\label{reconstruction}
\mf{a}_{N,\ep}
=
\mf{M}_{N,\ep}
\brac{ \phi,\widetilde\phi,\mu }^{-1}
\lim_{\lambda\to 0}
\mf{J}^\lambda_{N,\ep} 
\brac{ \phi,\widetilde\phi,\mu },
\quad
\mu\in (0,\widetilde\mu).
\end{align} 
\end{theorem}

We next mention 
some uniqueness and stability for identifying 
$\partial_\xi^\alpha \F V(0)$.
\begin{cor}\label{cor:main-2}
Let $j=1,2$. 
Suppose that $V_j$ satisfies 
(V1) with $V=V_j$ 
and that 
$V_1,V_2 \in B(M,L^1)$ 
for some $M>0$.
Let 
$S_j: B(\delta_0(M),H^1)\to B(C_0^2 \delta_0(M),H^1)$ 
be the scattering operator 
for the equation (\ref{eq:H}) with $V=V_j$.
Here, 
$C_0$ and 
$\delta_0(M)$ are positive numbers 
mentioned in Theorem \ref{thm:direct}.
Then the following properties hold:
\begin{enumerate}[(1)]
  \item 
Define $\nr{S_1 - S_2}$ by
\begin{align*}
\nr{S_1 - S_2}
=
\sup\brab{
\dfrac{\nr{(S_1 - S_2)\phi}_{H^1}
}{\nr{\phi}_{H^1}^3};\
\phi \in B(\delta_0(M);H^1)\setminus \{ 0\}
}.
\end{align*} 
Then   
for any $\alpha\in \Nz^3$,
we have 
\begin{align}\label{ineq:stability}
\ab{
\Pxi^\alpha \F V_1(0)
-
\Pxi^\alpha \F V_2(0)
}
\le
C
\braa{
\nr{S_1 - S_2}^{\frac{1}{|\alpha|+2}}
+
\nr{S_1 - S_2}
},
\end{align} 
where the constant 
$C$ is dependent on $|\alpha|,A$ and $M$. 
  \item 
If $S_1=S_2$, then $V_1=V_2$.
\end{enumerate}

\end{cor}

We enumerate some remarks for 
Theorem   \ref{thm:main-1} and 
Corollary \ref{cor:main-2}.

\begin{rem}\label{rem:phi}
Let $N\in \Nz$ and 
let $\ep\in \{ 0,1 \}^3$ with $|\ep| \neq 0$.
We now introduce an example of 
$\phi,\widetilde\phi\in \Sc\setminus \{ 0 \}$ 
satisfying (C1)--(C3).
Fix 
$\varphi_j\in \Sc(\R)\setminus\{ 0 \}$
($j=1,2,3$)
which are even.
We put  
\begin{align*}
\phi(x_1,x_2,x_3)
=
\varphi_1(x_1) \varphi_2(x_2) \varphi_3(x_3),
\quad
\widetilde\phi(x_1,x_2,x_3)
=
\Px^\ep \phi(x_1,x_2,x_3).
\end{align*} 
Then 
we immediately see that 
$\phi$ and $\widetilde\phi$ 
satisfy (C1) and (C2),
respectively. 
It follows from 
Proposition \ref{prop:C3} below 
that (C3) holds.
\end{rem}

\begin{rem}\label{rem:main}
We immediately see that 
for any $\beta \in \Nz^3$,
there uniquely exist 
$N\in \Nz$, $\ep\in \{ 0,1 \}^3$ and 
$k\in \{ 1,\cdots,N^\ast\}$ 
such that 
$\beta = 2\alpha(N,\ep;k)+\ep$.
Therefore,
using our main results, 
we can uniquely determine the exact value of 
$\Pxi^\beta \F V(0)$
for any $\beta \in \Nz^3$.
Furthermore,
it follows from Proposition \ref{prop:ana} that
\begin{align*}
\F V(\xi)
=
\sum_{|\beta|\ge 0}
\dfrac{\Pxi^\beta \F V(0)}{\beta !}
\xi^\beta,
\quad
|\xi| < A/3.
\end{align*} 
Using Proposition \ref{prop:ana} again and again,
we see the exact value of 
$\F V(\xi)$ ($\xi \in \R^3$)
and we can hence reconstruct $V$.
\end{rem}

Introducing        
the contents of the rest of this paper,
we close this section.
In Section \ref{sec:pre},
we show that 
the matrix 
$\mf{M}_{N,\ep}[\phi,\widetilde\phi,\mu]$
is invertible 
for some 
$\phi,\widetilde\phi\in \Sc \setminus \{ 0 \}$ 
and $\mu>0$.
In order to prove the invertibility,
we first show some propositions 
for functions 
$\Phi_{N,\ep}$ and $\Phi_\ep$.
Section \ref{sec:proof} is devoted to 
the proof of main results.
In Appendix \ref{sec:appendix},
we prove some supplementary propositions.

\section{Preliminaries}\label{sec:pre}
In this section,
we show propositions  
used in Section \ref{sec:proof} below.
In particular, 
we show the following lemma:
\begin{lem}\label{lem:M}
Let $N\in \Nz$ and $\ep\in \{ 0,1\}^3$.
Fix   
$\phi,\widetilde\phi \in \Sc\setminus \{ 0\}$.
Assume that 
$\phi=\widetilde\phi$
if $|\ep| = 0$
and that 
(C3) holds
if $|\ep| \neq 0$.
Then there exists some positive number $\widetilde\mu$ 
such that 
for any $\mu \in (0,\widetilde\mu)$,
the matrix 
$\mf{M}_{N,\ep}[\phi,\widetilde\phi,\mu]$ 
is invertible.
\end{lem}

To show Lemma \ref{lem:M}, 
we first prove some properties for functions  
$\Phi_{N,\ep}(\phi,\widetilde\phi,\mu;t,\xi)$
and 
$\Phi_\ep(\phi,\widetilde\phi;t,\xi)$.

\begin{prop}\label{prop:decay}
Let $N\in \Nz$ and $\ep\in\{ 0,1 \}^3$.
For any $\mu>0$ and $L>3$, 
there exists some positive number $C_L$ such that 
\begin{align*}
\ab{\Phi_{N,\ep}
\braa{\phi,\widetilde\phi,\mu;t,\xi
}}
\le 
\left\{
  \begin{array}{cl}
\dfrac{
C_L \bra{\xi}^{-2L}
}{
1+ t^2
\braa{\xi_1^2 + \mu^4 \xi_2^2}
} 
&\text{if $\ep =    (0,0,0)$,}\\[6mm]
\dfrac{
C_L \bra{\xi}^{-2L}
}{
1+ t^2 \xi_{m(\ep)}^2 
}    
&\text{if $\ep \neq (0,0,0)$ }\\
  \end{array}
\right.
\end{align*} 
for any 
$(t,\xi)\in \R\times \R^3$.
\end{prop}

\begin{proof}
It suffices to show the case $\ep =(0,0,0)$
because the other case $\ep \neq (0,0,0)$ 
can be proved more easily.
It follows that
\begin{align*}
&\F \braa{
\overline{U^\mu_{N,\ep}(t)}\phi \times
U^\mu_{N,\ep}(t) \widetilde\phi
}(\xi)
=
(2\pi)^{-3/2}
\braa{
\F \braa{\overline{U^\mu_{N,\ep}(t)}\phi}
\ast
\F \braa{U^\mu_{N,\ep}(t) \widetilde\phi}
}(\xi) \\
&=
(2\pi)^{-3/2}
\int_{\R^3}
\exp\braa{
it  \brab{
(\xi_1 - \eta_1)^2 + 
\mu^2     (\xi_2 - \eta_2)^2 
+
\mu^{2N+2}(\xi_3 - \eta_3)^2 
}}
\\
&\quad \times
\exp\braa{
-it  \brab{
\eta_1^2 + 
\mu^2 \eta_2^2 
+
\mu^{2N+2} \eta_3^2 
}}
\overline{\F^{-1}\phi(\xi - \eta)}\, 
\F \widetilde\phi(\eta)
d\eta 
\\
&=
(2\pi)^{-3/2}
\exp\braa{
it  \braa{
\xi_1^2 + 
\mu^2     \xi_2^2 
+
\mu^{2N+2} \xi_3^2 
}}
\\
&\quad \times
\int_{\R^3}
\exp\braa{
-2it  \braa{
\xi_1 \eta_1 + 
\mu^2 \xi_2 \eta_2
+
\mu^{2N+2} \xi_3 \eta_3 
}}
\overline{\F^{-1}\phi(\xi - \eta)}\, 
\F \widetilde\phi(\eta)
d\eta .
\end{align*} 
Let 
\begin{align*}
\Psi(t,\xi)
=
\int_{\R^3}
\exp\braa{
-2it  \braa{
\xi_1 \eta_1 + 
\mu^2 \xi_2 \eta_2
+
\mu^{2N+2} \xi_3 \eta_3 
}}
\overline{\F^{-1}\phi(\xi - \eta)}\, 
\F \widetilde\phi(\eta)
d\eta .
\end{align*} 
Since $\phi,\widetilde\phi \in \Sc$, 
we obtain for any $L>3$,
\begin{align}\label{ineq:psi-0}
\int_{\R^3}
\ab{ \F^{-1}\phi(\xi - \eta) }
\, 
\ab{ \F \widetilde\phi(\eta) }
d\eta 
\le C
\int_{\R^3}
\bra{\xi - \eta}^{-L}
\bra{\eta}^{-L}
d\eta
\le C
\bra{\xi}^{-L}
\end{align} 
and we hence see that 
\begin{align}\label{ineq:psi-1}
|\Psi(t,\xi)|
\le C
\bra{\xi}^{-L}.
\end{align}  
Furthermore,
by using integration by parts 
with respect to $\eta_1$,
we have  
\begin{align*}
\Psi(t,\xi)
&=
\dfrac{1}{2it\xi_1}
\int_{\R^3}
\exp\braa{
-2it  \braa{
\xi_1 \eta_1 + 
\mu^2 \xi_2 \eta_2
+
\mu^{2N+2} \xi_3 \eta_3 
}} \\
&\quad \times
\brab{
\overline{-\partial_1 \F^{-1}\phi(\xi - \eta)}\, 
\F \widetilde\phi(\eta)
+
\overline{\F^{-1}\phi(\xi - \eta)}\, 
\partial_1 \F \widetilde\phi(\eta)
}
d\eta .
\end{align*} 
From (\ref{ineq:psi-0}),
we see that 
\begin{align}\label{ineq:psi-2}
|\Psi(t,\xi)|
\le
\dfrac{C\bra{\xi}^{-L}}{|t\xi_1|}.
\end{align}  
Using integration by parts with respect to $\eta_2$,
we have  
\begin{align}\label{ineq:psi-3}
|\Psi(t,\xi)|
\le
\dfrac{C\bra{\xi}^{-L}}{|t\mu^2 \xi_2|}.
\end{align}
By (\ref{ineq:psi-1})--(\ref{ineq:psi-3}),
we complete the proof.

\end{proof}

\begin{prop}\label{prop:conv}
Let $N\in \Nz$ and $\ep\in\{ 0,1 \}^3$.
For any 
$(t,\xi)\in \R \times \R^3$,
we have 
\begin{align}\label{conv:psi}
\lim_{\mu \to 0}
\Phi_{N,\ep}(\phi,\widetilde\phi,\mu;t,\xi)
=
\Phi_\ep (\phi,\widetilde\phi;t,\xi).
\end{align} 
\end{prop}

\begin{proof}
For any 
$(t,\xi)\in \R \times \R^3$,
we obtain
\begin{align*}
&
\left|
\F \braa{
\overline{U^\mu_{N,\ep}(t)\phi} \times
U^\mu_{N,\ep}(t) \widetilde\phi
}(\xi)
-
\F \braa{
\overline{U_\ep(t)\phi} \times
U_\ep(t) \widetilde\phi
}(\xi)
\right|
\\
&\quad\le 
\left|
\F
\braa{
\braa{
\overline{U^\mu_{N,\ep}(t)\phi
-
U_\ep(t)\phi
}} \times
U^\mu_{N,\ep}(t) \widetilde\phi
}(\xi)
\right|
\\
&\qquad +
\left|
\F \braa{
\overline{U_\ep(t)\phi}
\times
\braa{
U^\mu_{N,\ep}(t)\widetilde\phi
-
U_\ep(t)\widetilde\phi
}
}(\xi)
\right| 
\\
&\quad\le C
\left\|
\braa{
\overline{U^\mu_{N,\ep}(t)\phi
-
U_\ep(t)\phi
}} \times
U^\mu_{N,\ep}(t) \widetilde\phi
\right\|_1
+ C
\left\|
\overline{U_\ep(t)\phi}
\times
\braa{
U^\mu_{N,\ep}(t)\widetilde\phi
-
U_\ep(t)\widetilde\phi
}
\right\|_1
\\
&\quad\le C
\left\|
U^\mu_{N,\ep}(t)\phi
-
U_\ep(t)\phi
\right\|_2
\,
\left\|
U^\mu_{N,\ep}(t) \widetilde\phi
\right\|_2
+ C
\left\|
U_\ep(t)\phi
\right\|_2
\, 
\left\|
U^\mu_{N,\ep}(t)\widetilde\phi
-
U_\ep(t)\widetilde\phi
\right\|_2  \\
&\quad\le C
\left\|
U^\mu_{N,\ep}(t)\phi
-
U_\ep(t)\phi
\right\|_2
\,
\left\|
\widetilde\phi
\right\|_2
+ C
\left\|
\phi
\right\|_2
\, 
\left\|
U^\mu_{N,\ep}(t)\widetilde\phi
-
U_\ep(t)\widetilde\phi
\right\|_2 . 
\end{align*} 
Since we have for any $\varphi\in \Sc$ and $t\in \R$,
\begin{align*}
\exp\braa{-it\ab{ D^{1,\mu}_{N,\ep}\xi }^2} 
\F \varphi(\xi)
\to
\exp\braa{-it \xi_{m(\ep)}^2} 
\F \varphi(\xi)
\
\text{
in $L^2(\R^3_\xi)$
as $\mu\to 0$,
}
\end{align*} 
we see that (\ref{conv:psi}) holds.
\end{proof}

\begin{prop}\label{prop:limit}
Let $N\in \Nz$ and $\ep\in\{ 0,1 \}^3$.
For any $\alpha\in \Nz^3$ with 
$\alpha_{m(\ep)}\neq 0$,
functions
$\xi^\alpha \Phi_{N,\ep}(\phi,\widetilde\phi,\mu;t,\xi)$ 
($\mu>0$) 
and 
$\xi^\alpha \Phi_\ep (\phi,\widetilde\phi;t,\xi)$
are integrable on 
$\R_t\times \R^3_\xi$.
Furthermore, 
we have 
\begin{align}\label{conv:limit}
\lim_{\mu\to 0}
\int_{\R^{1+3}}
\xi^\alpha 
\Phi_{N,\ep}(\phi,\widetilde\phi,\mu;t,\xi)
d(t,\xi)
=
\int_{\R^{1+3}}
\xi^\alpha 
\Phi_\ep (\phi,\widetilde\phi;t,\xi)
d(t,\xi). 
\end{align} 
\end{prop}

\begin{proof}
Let $L>|\alpha|+3$.
It follows from Proposition \ref{prop:decay} that 
\begin{align*}
\ab{
\xi^\alpha
\Phi_{N,\ep}(\phi,\widetilde\phi,\mu;t,\xi)
}
\le
\dfrac{C\bra{\xi}^{-2L}
\left| \xi^\alpha \right|
}{1+t^2 \xi_{m(\ep)}^2}.
\end{align*} 
Since $\alpha_{m(\ep)}\neq 0$,
we have
\begin{align*}
\int_{\R^{1+3}}
\dfrac{\bra{\xi}^{-2L}
\left| \xi^\alpha \right|
}{1+t^2 \xi_{m(\ep)}^2}
d(t,\xi)
&\le C
\int_{\R}
\dfrac{dt}{1+t^2}
\cdot
\int_{\R^{1+3}}
\bra{\xi}^{-2L}
|\xi^\alpha|
\,
|\xi_{m(\ep)}|^{-1}
d\xi 
\\
&\le C
\int_{\R^{1+3}}
\bra{\xi}^{-|\alpha|-7}
d\xi 
\\
&< \infty.
\end{align*} 
Thus,
for any $\mu>0$, 
the function 
$\xi^\alpha \Phi_{N,\ep}(\phi,\widetilde\phi,\mu;t,\xi)$ 
is integrable on 
$\R_t\times \R^3_\xi$.
The integrability of 
$\xi^\alpha \Phi_\ep (\phi,\widetilde\phi;t,\xi)$
is shown analogously.
We see from 
Proposition \ref{prop:conv} 
and 
the Lebesgue dominated convergence theorem
that 
(\ref{conv:limit}) holds.
\end{proof}

\begin{prop}\label{prop:log}
Let $N\in \Nz$ and $\ep = (0,0,0)$. 
Assume that $\phi\in \Sc \setminus \{ 0 \}$. 
Let $m,l\in \Nz$.
Then there exist 
positive numbers 
$\mu_1$,
$C_1$ and $C_2$ such that  
for any $\mu\in (0,\mu_1)$,
\begin{align}\label{ineq:log}
C_1 
\le 
\int_{\R^{1+3}}
\xi_2^{2m} \xi_3^{2l}
\Phi_{N,\ep}(\phi,\phi,\mu;t,\xi)
d(t,\xi)
\le
C_2
\braa{1+\left| \log\mu \right|}.
\end{align} 
\end{prop}
\begin{proof}
Since 
the function 
$\Phi_\ep(\phi,\phi;t,\xi)$ 
is continuous on $\R\times \R^3$, 
non-negative and 
does not equal to a zero function,
we see that 
\begin{align*}
C_3 :=
\int_{\R^{1+3}}
\xi_2^{2m} \xi_3^{2l}
\Phi_\ep(\phi,\phi;t,\xi)
d(t,\xi)
\in (0,\infty].
\end{align*} 
It follows from 
Proposition \ref{prop:conv} and Fatou's lemma
that 
\begin{align*}
C_3 
&=
\int_{\R^{1+3}}
\xi_2^{2m} \xi_3^{2l}
\liminf_{\mu\to 0}
\Phi_{N,\ep}(\phi,\phi,\mu;t,\xi)
d(t,\xi) 
\\
&\le
\liminf_{\mu\to 0}
\int_{\R^{1+3}}
\xi_2^{2m} \xi_3^{2l}
\Phi_{N,\ep}(\phi,\phi,\mu;t,\xi)
d(t,\xi) .
\end{align*} 
Therefore, 
there exists some $\mu_1 \in (0,1)$ 
such that 
for any $\mu\in (0,\mu_1)$,
\begin{align}\label{ineq:lower}
\min\left\{  
\dfrac{C_3}{2},\,
1
\right\}
\le
\int_{\R^{1+3}}
\xi_2^{2m} \xi_3^{2l}
\Phi_{N,\ep}(\phi,\phi,\mu;t,\xi)
d(t,\xi) .
\end{align} 

Henceforth,
we suppose that $\mu \in (0,\mu_1)$.
Let $L$ be a sufficiently large positive number. 
By Proposition \ref{prop:decay},
we have 
\begin{align*}
\int_{\R^{1+3}}
&
\xi_2^{2m} \xi_3^{2l}
\Phi_{N,\ep}(\phi,\phi,\mu;t,\xi)
d(t,\xi) 
\\
&\le C
\int_{\R^2}
\xi_2^{2m} \xi_3^{2l}
\braa{
\int_{\R^{1+1}}
\dfrac{\bra{\xi}^{-2L}}{
1+t^2\braa{\xi_1^2 + \mu^4 \xi_2^2}
}
d(t,\xi)
}
d(\xi_2,\xi_3)
\\
&\le C
\int_{\R^2}
\dfrac{\xi_2^{2m} \xi_3^{2l}}{
\braa{1+\xi_2^2 + \xi_3^2}^{L/2}}
\braa{
\int_\R
\braa{
\int_\R
\dfrac{\bra{\xi_1}^{-L}}{
1+t^2\braa{\xi_1^2 + \mu^4 \xi_2^2}
}
dt
}
d\xi_1
}
d(\xi_2,\xi_3).
\end{align*} 
Furthermore, we obtain 
\begin{align*}
\int_{|\xi_1|\ge 1}
\braa{
\int_\R
\dfrac{\bra{\xi_1}^{-L}}{
1 + t^2
\braa{\xi_1^2 + \mu^4\xi_2^2}
}
dt
}
d\xi_1
& \le
\int_{|\xi_1|\ge 1}
\bra{\xi_1}^{-L} |\xi_1|^{-1}
d\xi_1
\cdot
\int_\R
\dfrac{dt}{1+t^2} \\
& \le
\pi
\int_\R
\bra{\xi_1}^{-L}
d\xi_1 
\end{align*} 
and 
\begin{align*}
\int_{|\xi_1|< 1}
&
\braa{
\int_\R
\dfrac{\bra{\xi_1}^{-L}}{
1 + t^2
\braa{\xi_1^2 + \mu^4\xi_2^2}
}
dt
}
d\xi_1  
\\
& \le
\int_{-1}^1
\braa{
\int_\R
\dfrac{dt}{
1 + t^2
\braa{\xi_1^2 + \mu^4\xi_2^2}
}}
d\xi_1
=
\int_{-1}^1
\dfrac{d\xi_1}{
\sqrt{
\xi_1^2 + \mu^4 \xi_2^2
}}
\cdot
\int_\R
\dfrac{dt}{1+t^2}
\\
&=
\pi
\log\braa{\braa{
1+\sqrt{1+\mu^4 \xi_2^2}
}^2}
-
2\pi \log \mu 
-
\pi \log |\xi_2|
\\
&\le
2\pi 
\log\braa{
1+\sqrt{1+\xi_2^2}
}
-
2\pi \log \mu 
-
\pi \log |\xi_2| .
\end{align*}
Thus, we see that
\begin{align}\label{ineq:upper}
&
\int_{\R^{1+3}}
\xi_2^{2m} \xi_3^{2l}
\Phi_{N,\ep}(\phi,\phi,\mu;t,\xi)
d(t,\xi)  
\nonumber
\\
&\le C
\int_{\R^2}
\dfrac{\xi_2^{2m} \xi_3^{2l}}{
\braa{1+\xi_2^2 + \xi_3^2}^{L/2}}
\brab{
1
+
\log\braa{
1+\sqrt{1+\xi_2^2}}
+
\big| 
\log |\xi_2|
\big|
-
\log\mu
}
d(\xi_2,\xi_3)
\nonumber
\\
&\le
C\braa{
1+ | \log \mu |
}.
\end{align} 
From (\ref{ineq:lower}) and (\ref{ineq:upper}),
we have (\ref{ineq:log}).
\end{proof}

We next prove Lemma \ref{lem:M}.
If $N=0$, then we easily see that the lemma holds.
For other cases, we show dividing three steps.\\
(Step I)  
Let $N\in \N$ and $\ep\in\{ 0,1 \}^3$.
Then for any $2\le M \le N^\ast$, 
we have 
\begin{align*}
\sum_{j=1}^M 
(N+1)^{2(j-1)}
P_{N,\ep}
&
\braa{
\alpha(N,\ep;N^\ast + 1 -j)}
\\
&\le
(N+1)^{2(M-1)}
P_{N,\ep}
\braa{
\alpha(N,\ep;N^\ast + 1 -M)}
\\
&\quad +
\sum_{j=1}^{M-1}
(N+1)^{2(j-1)}
P_{N,\ep}\braa{
\alpha(N,\ep;N^\ast + 1 -j)}
\\
&\le
(N+1)^{2(M-1)}
P_{N,\ep}
\braa{
\alpha(N,\ep;N^\ast + 1 -M)}
\\
&\quad +
\dfrac{
(N+1)^{2(M-1)}-1
}{
(N+1)^2 -1 
}
\times N(N+1)
\\
&\le
(N+1)^{2(M-1)}
\brab{
P_{N,\ep}
\braa{
\alpha(N,\ep;N^\ast + 1 -M)}
+
\dfrac{N(N+1)}{N(N+2)}
}
\\
&\le
(N+1)^{2(M-1)}
P_{N,\ep}
\braa{
\alpha(N,\ep;N^\ast + 2 -M)}.
\end{align*}
(Step II)                 
Fix $N\in \N$.
Let $\mf{S}$ be the symmetric group of degree $N^\ast$.
We define $\sigma_0 \in \mf{S}$ by 
\begin{align*}
\sigma_0(j) = N^\ast +1-j,
\quad
j=1,\cdots,N^\ast .
\end{align*} 
Furthermore,
we define
\begin{align*}
Q_{N,\ep}(j,k)
&=
2(N+1)^{2(j-1)}
P_{N,\ep}\braa{
\alpha(N,\ep;k)},
\\
Q_{N,\ep}
&=
\sum_{j=1}^{N^\ast}
Q_{N,\ep}(j,\sigma_0(j)),
\\
\widetilde Q_{N,\ep}
&=
\inf_{\sigma \in \mf{S}\setminus\{ \sigma_0 \}}
\sum_{j=1}^{N^\ast}
Q_{N,\ep}(j,\sigma(j))
\quad
\text{if $N\ge 1$.}
\end{align*} 
Let $\sigma \in \mf{S}\setminus\{ \sigma_0 \}$.
Then 
there exists some 
$M\in \brab{1,\cdots,N^\ast }$ 
such that
$\sigma(M)>\sigma_0(M) = N^\ast + 1 - M$,
and such that 
if $M<N^\ast$ then 
$\sigma(j)=\sigma_0(j)$
($M+1\le j \le N^\ast$).
Therefore,
we obtain 
\begin{align*}
Q_{N,\ep}
&=
\sum_{j=1}^M
Q_{N,\ep}(j,\sigma_0(j))
+
\sum_{j=M+1}^{N^\ast}
Q_{N,\ep}(j,\sigma_0(j))
\\
&\le
2\sum_{j=1}^M 
(N+1)^{2(j-1)}
P_{N,\ep}\braa{
\alpha(N,\ep;N^\ast + 1 -j)}
+
\sum_{j=M+1}^{N^\ast}
Q_{N,\ep}(j,\sigma(j)),
\end{align*} 
where
a term 
$\sum_{j=N^\ast + 1}^{N^\ast} a_j$
is understood as 0.
We see from 
Step I 
and 
$\sigma(M)\ge N^\ast +2 -M$
that 
\begin{align*}
Q_{N,\ep}
&<
2 
(N+1)^{2(M-1)}
P_{N,\ep}
\braa{
\alpha(N,\ep;N^\ast + 2 -M)}
+
\sum_{j=M+1}^{N^\ast}
Q_{N,\ep}(j,\sigma(j))
\\
&\le
2\sum_{j=1}^M
(N+1)^{2(j-1)}
P_{N,\ep}
\braa{
\alpha(N,\ep;\sigma(j))}
+
\sum_{j=M+1}^{N^\ast}
Q_{N,\ep}(j,\sigma(j))
\\
&=
\sum_{j=1}^{N^\ast}
Q_{N,\ep}(j,\sigma(j)).
\end{align*} 
Therefore,
we have 
$Q_{N,\ep} < \widetilde{Q}_{N,\ep}$.\\
(Step III)      
Fix $N\in \N$.
For $\mu>0$, 
we set 
\begin{align*}
C_{N,\ep}(j,k;\mu)
&=
\int_{\R^{1+3}}
\xi^{2\alpha(N,\ep;k)+\ep}
\Phi_{N,\ep}\braa{
\phi,\widetilde\phi,\mu^{(N+1)^{2(j-1)}};
t,\xi
}
d(t,\xi),
\quad
1\le j,k \le N^\ast,
\\
A(\mu)
&=
\braa{
\mu^{Q_{N,\ep}(j,k)}
C_{N,\ep}(j,k;\mu)
}_{1\le j,k \le N^\ast}.
\end{align*} 
Then it suffices to show that 
for some $\widetilde\mu>0$,
we have  
$\det A(\mu) \neq 0$ 
($\mu\in (0,\widetilde\mu)$).

We see from 
the assumptions of $\phi$ and $\widetilde\phi$, 
Propositions \ref{prop:limit} and \ref{prop:log}
that
there exist some  
$\mu_1 \in (0,1)$ and  $C_1,C_2>0$
such that
for any $\mu \in (0,\mu_1)$,
\begin{align*}
C_1 
\le 
\left|
C_{N,\ep}(j,k;\mu)
\right| 
\le
C_2
\braa{
1+| \log\mu |
}.
\end{align*} 
Therefore, 
for any $\mu\in (0,\mu_1)$,
it follows from Step II that  
\begin{align*}
&\left| \det A(\mu) \right| 
\\
&\ge
\prod_{j=1}^{N^\ast}
\mu^{Q_{N,\ep}(j,\sigma_0(j))}
\left| C_{N,\ep}(j,\sigma_0(j);\mu) \right| 
 -
\sum_{\sigma\in \mf{S}\setminus \{ \sigma_0\}}
\prod_{j=1}^{N^\ast}
\mu^{Q_{N,\ep}(j,\sigma(j))}
\left| C_{N,\ep}(j,\sigma(j);\mu) \right| 
\\
&\ge
C_1^{N^\ast}\mu^{Q_{N,\ep}}
-
N^\ast !\,
C_2^{N^\ast}
\braa{1+| \log\mu |}^{N^\ast}
\mu^{\widetilde{Q}_{N,\ep}}
\\
&\ge
\braa{
C_1^{N^\ast}
-
N^\ast !\,
C_2^{N^\ast}
\braa{1+| \log\mu |}^{N^\ast}
\mu
}
\mu^{Q_{N,\ep}}.
\end{align*} 
Since 
\begin{align*}
\lim_{\mu\to 0}
\braa{1+| \log\mu |}^{N^\ast}\mu
= 0,
\end{align*} 
there exists $\widetilde\mu>0$ such that 
\begin{align*}
C_1^{N^\ast}
-
N^\ast !\,
C_2^{N^\ast}
\braa{1+| \log\mu |}^{N^\ast}
\mu
>0,
\quad
\mu\in (0,\widetilde\mu),
\end{align*} 
which completes the proof.

\section{Proof of Main Theorems}\label{sec:proof}
In this section,
we prove main results.
Throughout this section, 
we fix $N\in \Nz$, $\ep\in \{ 0, 1 \}^3$ and              
$\phi,\widetilde\phi\in \Sc \setminus\{ 0 \}$.
Furthermore, 
we assume that 
$\phi = \widetilde\phi$ and 
(C1) holds if $|\ep|=0$,  
and assume that 
(C1)--(C3) hold if $|\ep|\neq 0$. 
Fix $\mu\in (0,\widetilde\mu)$,
where $\widetilde\mu$ is the positive number 
mentioned in Lemma \ref{lem:M}.
Then the matrix 
$\mf{M}_{N,\ep}[\phi,\widetilde\phi,\mu]$ 
is invertible.
We remark that 
all positive constants $C$ 
which appear in this section 
are independent of the positive parameter $\lambda$.

\subsection{Proof of Theorem \ref{thm:main-1}}
We assume that            
the interaction potential $V$ 
of the equation (\ref{eq:H})
satisfies (V1).
We see from Theorem \ref{thm:direct}(3) 
that 
for any $\psi \in \Sc$, 
the scattering operator $S$
is expressed by 
\begin{align*}
S(\vep \psi)
=
\vep\psi
- i
\int_{\R}
U(-t)F(u_\vep(t))
dt,
\end{align*} 
where
$\vep$ is sufficiently small 
and 
$u_\vep$ is the time-global solution 
to (\ref{eq:H}) satisfying  
\begin{align*}
\lim_{t\to -\infty}
\nr{ u(t) - U(t)(\vep\psi) }_{H^1}
=
0.
\end{align*} 
By the inequality (\ref{ineq:u - U phi}),
we have 
\begin{align*}
&
\ab{
i\vep^{-3}
\bra{(S-\id)(\vep\psi),\widetilde\psi}
-
(2\pi)^{3/2}
\int_{\R^{1+3}}
\F V(\xi)
\Phi_{N,\ep}(\psi,\widetilde\psi,1;t,\xi)
d(t,\xi)
}
\\
&\le 
C
\vep^2 \| \psi \|_{H^1}^5 \nr{\widetilde\psi}_{H^1}.
\end{align*} 
Substituting 
$\psi = \phi\circ D_{N,\ep}^{\lambda,\mu}$,
$\widetilde\psi = \widetilde\phi\circ D_{N,\ep}^{\lambda,\mu}$ 
and 
$\vep = \lambda^{N+4}$
for the above inequality,
we obtain 
for sufficiently small $\lambda>0$,
\begin{align*}
&
\Bigg|
i\lambda^{-3N-12}
\bra{(S-\id)(
\lambda^{N+4}
\phi\circ D_{N,\ep}^{\lambda,\mu}),
\widetilde\phi\circ D_{N,\ep}^{\lambda,\mu}}
\\
&\quad -
(2\pi)^{3/2}
\int_{\R^{1+3}}
\F V(\xi)
\Phi_{N,\ep}(
\phi\circ D_{N,\ep}^{\lambda,\mu},
\widetilde\phi\circ D_{N,\ep}^{\lambda,\mu},
1;t,\xi)
d(t,\xi)
\Bigg|
\\
&\le 
C
\lambda^{2N+8} 
\nr{          \phi\circ D_{N,\ep}^{\lambda,\mu}}_{H^1}^5 
\nr{\widetilde\phi\circ D_{N,\ep}^{\lambda,\mu}}_{H^1}
\\
&\le
C
\lambda^{2N+8} \lambda^{-\frac{3}{2}\cdot 6}
\nr{          \phi}_{H^1}^5 
\nr{\widetilde\phi}_{H^1}.
\end{align*} 
Since                           
\begin{align*}
&
\int_{\R^{1+3}}
\F V(\xi)
\Phi_{N,\ep}(
\phi\circ D_{N,\ep}^{\lambda,\mu},
\widetilde\phi\circ D_{N,\ep}^{\lambda,\mu},
1;t,\xi)
d(t,\xi)
\\
&=
\lambda^{-5}\mu^{-N-2}
\int_{\R^{1+3}}
\F V(D_{N,\ep}^{\lambda,\mu}\xi)
\Phi_{N,\ep}(
\phi,
\widetilde\phi,
\mu;t,\xi)
d(t,\xi),
\end{align*} 
we have 
\begin{align*}
&
\Bigg|
\dfrac{i\lambda^{-3N-7}\mu^{N+2}}{(2\pi)^{3/2}}
\bra{(S-\id)(
\lambda^{N+4}
\phi\circ D_{N,\ep}^{\lambda,\mu}),
\widetilde\phi\circ D_{N,\ep}^{\lambda,\mu}}
\\
& \qquad\qquad -
\int_{\R^{1+3}}
\F V(D_{N,\ep}^{\lambda,\mu}\xi)
\Phi_{N,\ep}(
\phi,
\widetilde\phi,
\mu;t,\xi)
d(t,\xi)
\Bigg|
\\
&\le
C
\lambda^{2N+4} 
\nr{          \phi}_{H^1}^5 
\nr{\widetilde\phi}_{H^1}.
\end{align*} 
Therefore,
it follows that 
\begin{align}
&
\Bigg|
\dfrac{i\mu^{N+2}}{(2\pi)^{3/2}}
\lambda^{-2N-|\ep|}
\vD_\lambda^{2N+|\ep|}
\brab{
\lambda^{-3N-7}
\bra{(S-\id)(
\lambda^{N+4}
\phi\circ D_{N,\ep}^{\lambda,\mu}),
\widetilde\phi\circ D_{N,\ep}^{\lambda,\mu}}
}
\nonumber
\\
& \qquad\qquad\qquad -
\lambda^{-2N-|\ep|}
\vD_\lambda^{2N+|\ep|}
\int_{\R^{1+3}}
\F V(D_{N,\ep}^{\lambda,\mu}\xi)
\Phi_{N,\ep}(
\phi,
\widetilde\phi,
\mu;t,\xi)
d(t,\xi)
\Bigg|
\nonumber
\\
&\le
C
\lambda \label{ineq:pr-thm-1}
\end{align} 
for sufficiently small $\lambda>0$.

We now define          
\begin{align*}
f(\lambda)
=
\int_{\R^{1+3}}
\F V(D_{N,\ep}^{\lambda,\mu}\xi)
\Phi_{N,\ep}(
\phi,
\widetilde\phi,
\mu;t,\xi)
d(t,\xi).
\end{align*} 
Since $V$ satisfies the condition (V1),
we see from Proposition \ref{prop:decay} 
that 
$f(\lambda)$ is smooth on $(0,\lambda_0)$ 
for some $\lambda_0>0$.
Using Proposition \ref{prop:vD} below,
we obtain 
\begin{align}\label{ineq:pr-thm-2}
\ab{
\lambda^{-2N-|\ep|}
\vD_\lambda^{2N+|\ep|}
f(\lambda)
-
\dfrac{\partial^{2N+|\ep|}}{
\partial \lambda^{2N+|\ep|}}
f(\lambda)
} 
\le 
C\lambda
\end{align} 
for any $\lambda\in (0,\lambda_0/(2N+4))$.
Furthermore, 
we see that      
\begin{align*}
\dfrac{\partial^{2N+|\ep|}}{
\partial \lambda^{2N+|\ep|}}
f(\lambda)
=
\sum_{|\beta|=2N+|\ep|}
\dfrac{|\beta|!}{\beta !}
\mu^{P_{N,\ep}(\beta)}
\int_{\R^{1+3}}
 \xi^\beta
\Pxi^\beta
\F V(D_{N,\ep}^{\lambda,\mu}\xi)
\Phi_{N,\ep}(
\phi,
\widetilde\phi,
\mu;t,\xi)
d(t,\xi).
\end{align*} 
For any $\lambda>0$,    
$\beta\in \Nz^3$ and $\xi\in \R^3$,
we see from the condition (V1) that 
\begin{align*}
&
\ab{
\Pxi^\beta \F V(D_{N,\ep}^{\lambda,\mu}\xi)
-
\Pxi^\beta \F V(0)
} 
\\
&\le
(2\pi)^{-3/2}
\nr{
x^\beta V(x) 
\braa{
\exp\braa{-iD_{N,\ep}^{\lambda,\mu}\xi \cdot x} -1
}
}_{L^1(\R^3_x)}
\\
&\le
(2\pi)^{-3/2}
\nr{
e^{-A|x|/2} 
\braa{
\exp\braa{-iD_{N,\ep}^{\lambda,\mu}\xi \cdot x} -1
}
}_{L^\infty(\R^3_x)}
\nr{
x^\beta e^{A|x|/2} V(x) 
}_{L^1(\R^3_x)}
\\
&\le
C
|\lambda \xi| .
\end{align*} 
Hence we see that 
\begin{align*}
\ab{
\dfrac{\partial^{2N+|\ep|}}{
\partial \lambda^{2N+|\ep|}}
f(\lambda)
-
\sum_{|\beta|=2N+|\ep|}
\dfrac{|\beta|!}{\beta !}
\mu^{P_{N,\ep}(\beta)}
\Pxi^\beta
\F V(0)
\int_{\R^{1+3}}
 \xi^\beta
\Phi_{N,\ep}(
\phi,
\widetilde\phi,
\mu;t,\xi)
d(t,\xi)
}
\le
C\lambda
\end{align*} 
for any $\lambda\in (0,\lambda_0/(2N+4))$.
By conditions             
for $\phi$ and $\widetilde\phi$,
we have 
\begin{align*}
\int_{\R^{1+3}}
 \xi^{2\alpha + \widetilde\ep}
\Phi_{N,\ep}(
\phi,
\widetilde\phi,
\mu;t,\xi)
d(t,\xi)
= 0
\end{align*} 
for any $\alpha \in \Nz^3$ and 
for any 
$\widetilde\ep \in \{ 0,1 \}^3 \setminus \{ \ep \}$.
Therefore,
it follows that 
\begin{align}
&
\Bigg|
\dfrac{\partial^{2N+|\ep|}}{
\partial \lambda^{2N+|\ep|}}
f(\lambda)
-
\Bigg\{
\sum_{k=1}^{N\ast}
\dfrac{(2N+|\ep|)!}{\braa{2\alpha(N,\ep;k)+\ep}!}
\mu^{P_{N,\ep}\braa{2\alpha(N,\ep;k)+\ep}}
\nonumber
\\
&\quad \times
\Pxi^{2\alpha(N,\ep;k)+\ep}
\F V(0)
\int_{\R^{1+3}}
 \xi^{2\alpha(N,\ep;k)+\ep}
\Phi_{N,\ep}(
\phi,
\widetilde\phi,
\mu;t,\xi)
d(t,\xi)
\Bigg\}
\Bigg|
\nonumber
\\
&\le
C\lambda
\label{ineq:pr-thm-3}
\end{align} 
for any $\lambda\in (0,\lambda_0/(2N+4))$.

Therefore,      
we see from 
(\ref{ineq:pr-thm-1})--(\ref{ineq:pr-thm-3})
that 
there exists some positive number
$\widetilde\lambda$ such that
\begin{align*}
\ab{
\mf{J}^\lambda_{N,\ep} 
\left[ \phi,\widetilde\phi,\mu \right]
-
\mf{M}_{N,\ep}
\left[ \phi,\widetilde\phi,\mu \right]
\mf{a}_{N,\ep}
}
\le
C\lambda,
\quad
\lambda\in (0,\widetilde\lambda).
\end{align*} 
Since 
the matrix 
$\mf{M}_{N,\ep}[\phi,\widetilde\phi,\mu]$ 
is invertible,
the proof is complete.

\subsection{Proof of Corollary \ref{cor:main-2}}
Let $j=1,2$.
Suppose that $V_j$ satisfies 
(V1) with $V=V_j$.
Let $S_j$ be the scattering operator 
for the equation (\ref{eq:H}) with $V=V_j$.
We denote $\mf{a}_{N,\ep}$ 
(resp. 
$\mf{J}^\lambda_{N,\ep} [\phi,\widetilde\phi,\mu]$)
with $V=V_j$  
by $\mf{a}_{N,\ep}^j$
(resp. 
$\mf{J}^{\lambda,j}_{N,\ep} [\phi,\widetilde\phi,\mu]$).
We see from Theorem \ref{thm:main-1} that 
for some $\lambda_0>0$ and 
$\phi,\widetilde\phi\in \Sc$,
\begin{align*}
\ab{
\mf{a}_{N,\ep}^1 - \mf{a}_{N,\ep}^2
}
\le
\ab{
\mf{M}_{N,\ep}
\left[ \phi,\widetilde\phi,\mu \right]^{-1}
\braa{
\mf{J}^{\lambda,1}_{N,\ep} 
\left[ \phi,\widetilde\phi,\mu \right]
-
\mf{J}^{\lambda,2}_{N,\ep} 
\left[ \phi,\widetilde\phi,\mu \right]
}
}
+
C\lambda,
\quad
\lambda\in (0,\lambda_0).
\end{align*} 
Since
\begin{align*}
&
\ab{
\lambda^{-2N-|\ep|}
\vD_\lambda^{2N+|\ep|}
\brab{
\lambda^{-3N-7}
\bra{(S_1 - S_2)(
\lambda^{N+4}
\phi\circ D_{N,\ep}^{\lambda,\mu}),
\widetilde\phi\circ D_{N,\ep}^{\lambda,\mu}
}}}
\\
&\le
C
\lambda^{-2N-|\ep|}
\lambda^{-3N-7}
\nr{S_1 - S_2}
\nr{
\lambda^{N+4}
\phi\circ D_{N,\ep}^{\lambda,\mu}
}_{H^1}^3
\nr{
\widetilde\phi\circ D_{N,\ep}^{\lambda,\mu}
}_{H^1}
\\
&\le
C
\lambda^{-2N-|\ep|}
\lambda^{-3N-7}
\lambda^{3N+12}
\lambda^{-\frac{3}{2}\cdot 4}
\nr{S_1 - S_2}
\\
&\le
C
\lambda^{-2N-|\ep|-1}
\nr{S_1 - S_2},
\end{align*} 
we have
\begin{align*}
\ab{
\mf{a}_{N,\ep}^1 - \mf{a}_{N,\ep}^2
}
\le
C
\brab{
\lambda^{-2N-|\ep|-1}
\nr{S_1 - S_2} + \lambda
}.
\end{align*} 
Let 
$a=(2N+|\ep|+2)^{-1}$ and put 
\begin{align*}
\lambda = 
\min\brab{
\lambda_0,\, \nr{S_1 - S_2}^a}.
\end{align*} 
Then (\ref{ineq:stability}) holds.
By (\ref{ineq:stability}) and 
Proposition \ref{prop:ana} below,
we see that the property (2) holds.
We hence complete the proof.

\appendix
\section{Supplementary Propositions}\label{sec:appendix}
In this section,
we prove supplementary propositions
used above.
For $n\in \N$,
we define the Fourier transform $\F_n$ 
on $L^1(\R^n)$
by 
\begin{align*}
\F_n f(\xi) = 
(2\pi)^{-n/2}
\int_{\R^n} 
e^{-ix\cdot \xi} f(x)
dx,
\quad 
f\in L^1(\R^n), 
\ 
\xi\in \R^n. 
\end{align*}

We first show the analyticity of the Fourier transform of 
a function $f(x)$ 
which exponentially decreases.
\begin{prop}\label{prop:ana}
Let $n\in \N$ and let $A>0$.
Assume that  
a measurable function $f:\R^n\to\C$ 
satisfies   
$f(x)\exp(A|x|)\in L^1(\R_x^n)$.
Then $\F_n f(\xi)$ is analytic on $\R_{\xi}^n$.
Furthermore,
for any $\xi_0 \in \R^n$,
if $\xi\in \R^n$ satisfies  
$|\xi -\xi_0|<A/n$,  
then we have 
\begin{align*}
\F_n f(\xi)
=
\sum_{|\alpha| \ge 0}
\dfrac{\Pxi^\alpha \F_n f(\xi_0)
}{
\alpha!
}
(\xi -\xi_0)^\alpha
.
\end{align*} 
\end{prop}

\begin{proof}
Fix $n\in \N$, $\xi_0 \in \R^n$ and $\beta\in \Nz^n$.
Then we see that $\F_n f$ is smooth and that  
\begin{align*}
\ab{\Pxi^\beta \F_n f(\xi_0)}
=
\ab{\F_n (x^\beta f)(\xi_0)}
&\le
(2\pi)^{-n/2}
\int_{\R^n} \ab{x^\beta f(x)}
dx
\\
&\le
(2\pi)^{-n/2}
\nr{f e^{A|x|}}_{L^1(\R^n)}
\nr{x^\beta e^{-A|x|}}_{L^\infty(\R^n)}
\end{align*} 
and 
\begin{align*}
\nr{x^\beta e^{-A|x|}}_{L^\infty(\R^n)}
\le
\prod_{j=1}^n
\sup_{r\ge 0}
\braa{ r^{\beta_j} e^{-Ar/n} }
=
\prod_{j=1}^n
n^{\beta_j}
A^{-\beta_j} {\beta_j}^{\beta_j} e^{-\beta_j}
=
n^{|\beta|}
A^{-|\beta|} 
\prod_{j=1}^n
{\beta_j}^{\beta_j} e^{-\beta_j}.
\end{align*} 
It follows from Stirling's formula that 
\begin{align*}
\beta !
=
\prod_{j=1}^n
\beta_j !
\ge
\prod_{j=1}^n
{\beta_j}^{\beta_j} e^{-\beta_j},
\end{align*} 
and we hence obtain 
\begin{align*}
\ab{\dfrac{\Pxi^\beta \F_n f(\xi_0)
}{
\beta !
}}
\le
C_A
\braa{\dfrac{n}{A}}^{|\beta|},
\end{align*} 
where 
$
C_A
=
(2\pi)^{-n/2}
\nr{f e^{A|x|}}_{L^1(\R^n)}
$
is a positive constant independent of $\beta$.
Therefore,
for any 
$\xi \in \R^n$ with $|\xi -\xi_0|<A/n$,
we have 
\begin{align*}
\sum_{|\alpha|\ge 0}
\ab{\dfrac{\Pxi^\alpha \F_n f(\xi_0)
}{
\alpha!
}
(\xi -\xi_0)^\alpha
}
&\le
C_A
\sum_{|\alpha|\ge 0}
\brab{
\braa{\dfrac{n}{A}}
\ab{\xi -\xi_0}
}^{|\alpha|}
\\
&=
C_A
\sum_{m=0}^\infty
\sum_{|\alpha|= m}
\brab{
\braa{\dfrac{n}{A}}
\ab{\xi -\xi_0}
}^m
\\
&=
C_A
\sum_{m=0}^\infty
\#\brab{ \alpha\in \Nz^n;\ |\alpha|=m }
\times
\brab{
\braa{\dfrac{n}{A}}
\ab{\xi -\xi_0}
}^m
\\
&\le 
C_A
\sum_{m=0}^\infty
(m+n-1)^{n-1}
\times
\brab{
\braa{\dfrac{n}{A}}
\ab{\xi -\xi_0}
}^m
<
\infty.
\end{align*}
Let $N\in \N$.
Then we obtain 
\begin{align*}
&
\ab{
\sum_{0\le |\alpha| \le N}
\dfrac{\Pxi^\alpha \F_n f(\xi_0)
}{
\alpha!
}
(\xi -\xi_0)^\alpha
-
\F_n f(\xi)
} 
\\
&\le 
\sum_{|\alpha| = N+1}
\sup_{\theta \in [0,1]}
\ab{
\dfrac{\Pxi^\alpha 
\F_n f\braa{
(1-\theta) \xi + \theta \xi_0}
}{
\alpha!
}
(\xi -\xi_0)^\alpha
}
\\
& \le
C_A(N+n)^{n-1}
\brab{
\braa{\dfrac{n}{A}}
\ab{\xi -\xi_0}
}^{N+1}
\to 0
\quad
\text{as $N\to \infty$.}
\end{align*}
Therefore, 
we see that $\F_n f$ is analytic.
We hence complete the proof.
\end{proof}

In Remark \ref{rem:phi} above,
we introduce an example of data 
$\phi$ and $\widetilde\phi$ 
satisfying conditions (C1)--(C3).
We now prove that 
the example ($\phi$, $\widetilde\phi$) actually satisfies (C3).
\begin{prop}\label{prop:C3}
Let $\ep \in \{ 0,1 \}^3$ with $|\ep|\neq 0$.
Fix $\varphi_j \in \Sc(\R)\setminus \{ 0 \}$ 
($j=1,2,3$)
which are even.
Define $\phi(x)$ and $\widetilde\phi(x)$ 
($x=(x_1,x_2,x_3)\in \R^3$) by 
\begin{align*}
\phi(x_1,x_2,x_3)
=
\varphi_1(x_1) \varphi_2(x_2) \varphi_3(x_3)
\quad
\text{and}
\quad
\widetilde\phi(x)
=
\Px^\ep \phi(x),
\end{align*} 
respectively.
Then we have for any $\alpha\in \Nz^3$,
\begin{align}\label{neq}
\int_{\R^{1+3}}
\xi^{2\alpha + \ep}
\Phi_{\ep}\braa{\phi,\widetilde\phi;t,\xi}
d(t,\xi)
\neq
0.
\end{align} 
\end{prop}

\begin{proof}
We consider only the case $\ep = (1,0,1)$ 
because other cases are proved similarly.
Let $\rho, \widetilde\rho \in \Sc(\R)$.
Then we have the following properties:
\begin{enumerate}[(i)]
  \item 
For any $\zeta\in \R$,
it follows that 
$
\zeta \braa{\F_1 \rho}(\zeta) = 
-i \F_1\braa{\rho^\prime}(\zeta) 
$.
  \item 
If $\rho$ is odd, 
then 
$-i\F_1 \Rep \rho = \Imp \F_1 \rho$.
  \item 
If $\rho$ is even, 
then 
$\F_1 |\rho|^2 $
is a real-valued function.
\end{enumerate}
We define the function 
$\Psi(\rho,\widetilde\rho;t,\zeta)$
($t,\zeta\in \R$) by 
\begin{align*}
\Psi\braa{\rho,\widetilde\rho;t,\zeta}
=
\F_1\braa{ \ab{U_1(t)\rho}^2 }(\zeta)
\times
\overline{
\F_1\braa{ 
\overline{
U_1(t)\rho
}
\times
U_1(t)\widetilde\rho 
}(\zeta)},
\end{align*}  
where 
$U_1(t)=\F_1^{-1} e^{-it\zeta^2} \F_1$.
If $\rho$ is even,  
then it follows from (i)--(iii) that 
\begin{align*}
\zeta \Psi\braa{\rho,\rho;t,\zeta}
&=
\F_1\braa{ \ab{U_1(t)\rho}^2 }(\zeta)
\times
\overline{
\zeta
\F_1\braa{ 
\overline{
U_1(t)\rho
}
\times
U_1(t)\rho 
}(\zeta)}   
\\
&=
\F_1\braa{ \ab{U_1(t)\rho}^2 }(\zeta)
\times
\overline{
(-i)
\F_1\braa{ 
\overline{
U_1(t)\rho
}
\times
U_1(t)\rho 
}^\prime(\zeta)}
\\
&=
2
\F_1\braa{ \ab{U_1(t)\rho}^2 }(\zeta)
\times
\overline{
(-i)
\F_1
\Rep
\braa{ 
\overline{
U_1(t)\rho
}
\times
U_1(t)\rho^\prime 
}(\zeta)}
\\
&=
2
\F_1\braa{ \ab{U_1(t)\rho}^2 }(\zeta)
\times
\overline{
\Imp \F_1 
\braa{ 
\overline{
U_1(t)\rho
}
\times
U_1(t)\rho^\prime 
}(\zeta)}
\\
&=
-2 \Imp
\brab{
\F_1\braa{ \ab{U_1(t)\rho}^2 }(\zeta)
\times
\overline{
\F_1 
\braa{ 
\overline{
U_1(t)\rho
}
\times
U_1(t)\rho^\prime 
}(\zeta)}}
\\
&=
-2 \Imp
\Psi\braa{\rho,\rho^\prime;t,\zeta},
\quad
t,\zeta\in \R .
\end{align*} 
Hence we have 
\begin{align}\label{id:Psi}
\Imp \Psi\braa{\rho,\rho^\prime;t,\zeta}
=
-
\dfrac{1}{2}
\zeta \Psi\braa{\rho,\rho;t,\zeta},
\quad
t,\zeta\in \R.
\end{align} 
Let $m\in \Nz$. 
As in the proof of Proposition \ref{prop:decay},
we obtain
\begin{align*}
\Psi\braa{\rho,\rho^\prime;t,\zeta} \le
\dfrac{C\bra{\zeta}^{-m-2}}{1+t^2 \zeta^2},
\quad
(t,\zeta)\in \R\times\R .
\end{align*} 
Therefore, 
we see that 
$
\zeta^{2m+1}
\Psi\braa{\rho,\rho^\prime;t,\zeta}
$
and 
$
\zeta^{2m+1}
\Psi\braa{\rho,\rho^\prime;0,\zeta}
$
are integrable on 
$\R_t\times \R_\zeta$ and $\R_\zeta$,
respectively.
Furthermore,
if $\rho$ is nonzero,
then we have 
\begin{align}\label{nonzero}
\int_\R 
\zeta^{2m+2}
\Psi\braa{\rho,\rho;0,\zeta}
d\zeta 
>0.
\end{align} 
Therefore, 
if $\rho$ is even and nonzero,
then we see from 
(\ref{id:Psi}) and (\ref{nonzero}) that 
\begin{align}
\Imp
\int_{\R^{1+1}} 
\zeta^{2m+1}
\Psi\braa{\rho,\rho^\prime;t,\zeta}
d(t,\zeta) 
&=
-
\dfrac{1}{2}
\int_{\R^{1+1}} 
\zeta^{2m+2}
\Psi\braa{\rho,\rho;t,\zeta}
d(t,\zeta) 
\neq 0,\label{nonzero-1}\\
\Imp
\int_\R 
\zeta^{2m+1}
\Psi\braa{\rho,\rho^\prime;0,\zeta}
d\zeta
&=
-
\dfrac{1}{2}
\int_\R 
\zeta^{2m+2}
\Psi\braa{\rho,\rho;0,\zeta}
d\zeta 
\neq 0.\label{nonzero-2}
\end{align} 
Let 
$\alpha=(\alpha_1,\alpha_2,\alpha_3)\in \Nz^3$.
Then it follows from the equality 
\begin{align*}
\int_{\R^{1+3}}
&
\xi^{2\alpha + \ep}
\Phi_{\ep}\braa{\phi,\widetilde\phi;t,\xi}
d(t,\xi) 
\\
&=
\int_{\R^{1+1}} 
\xi_1^{2\alpha_1+1}
\Psi\braa{\varphi_1,\varphi_1^\prime;t,\xi_1}
d(t,\xi_1) 
\\
&\quad
\times
\int_\R 
\xi_2^{2\alpha_2}
\Psi\braa{\varphi_2,\varphi_2;0,\xi_2}
d\xi_2
\times
\int_\R 
\xi_3^{2\alpha_3+1}
\Psi\braa{\varphi_3,\varphi_3^\prime;0,\xi_3}
d\xi_3
\end{align*} 
and (\ref{nonzero})--(\ref{nonzero-2})
that (\ref{neq}) holds.
\end{proof}

Finally,
we show the following inequality 
used in the proof of Theorem \ref{thm:main-1}:
\begin{prop}\label{prop:vD}
Fix $\lambda_0>0$ and $L\in \Nz$.
Let 
$h\in C^{L+1}((0,\lambda_0))$.
If
\begin{align}\label{condi:bdd h}
\max_{l=0,\cdots,L+1}
\sup_{\lambda \in (0,\lambda_0)}
\ab{
\dfrac{d^l}{d \lambda^l}
h(\lambda)
}
<\infty,
\end{align} 
then we have 
\begin{align*}
\ab{
\dfrac{d^L}{d \lambda^L}h(\lambda)
-
\lambda^{-L}\vD_\lambda^L h(\lambda)
}
\le C\lambda,
\quad
\lambda \in 
\braa{ 0,\dfrac{\lambda_0}{L+1}}
\end{align*} 
for some positive number $C$ 
independent of $\lambda$.
\end{prop}

\begin{proof}
Fix $\lambda\in (0,\lambda_0/(L+1))$. 
Then 
we see from 
Taylor's theorem 
and (\ref{condi:bdd h})
that for any $l=0,1,\cdots,L$,
\begin{align*}
\ab{
h((l+1)\lambda )
-
\sum_{k=0}^L 
\frac{(l\lambda)^k}{k!}
h^{(k)}(\lambda)
}
\le
C\lambda^{L+1}.
\end{align*}  
Since 
\begin{align*}
\sum_{l=0}^L
\frac{(-1)^{L-l}L!}{l! (L-l)!}
\sum_{k=0}^L 
\frac{(l\lambda)^k}{k!}
h^{(k)}(\lambda)
=
L!
\sum_{k=0}^L 
\left(
\sum_{l=0}^L
\frac{(-1)^{L-l}l^k}{l! (L-l)!}
\right)
\frac{\lambda^k}{k!}
h^{(k)}(\lambda),
\end{align*} 
we obtain 
\begin{align*}
\ab{
\varDelta_\lambda^L h(\lambda)
-
L!
\sum_{k=0}^L 
\left(
\sum_{l=0}^L
\frac{(-1)^{L-l}l^k}{l! (L-l)!}
\right)
\frac{\lambda^k}{k!}
h^{(k)}(\lambda)
}
\le
C \lambda^{L+1}.
\end{align*} 
Therefore, 
it suffices to prove the following equality:
\begin{align}\label{magic}
\sum_{l=0}^L
\frac{(-1)^{L-l}l^k}{l! (L-l)!}
=
\left\{
  \begin{array}{ll}
0&\text{if $k=0,1,\cdots,L-1$},\\
1&\text{if $k=L$}.\\
  \end{array}
\right.
\end{align}
Let $g(y)=y^L$ ($y\in \R$).
Then the $L$-th forward difference of $g$ 
satisfies
\begin{align*}
\sum_{l=0}^L
\frac{(-1)^{L-l}L!}{l! (L-l)!}
g(y+l)
=
L!
\end{align*} 
and 
\begin{align*}
\sum_{l=0}^L
\frac{(-1)^{L-l}L!}{l! (L-l)!}
g(y+l)
=
L!
\sum_{k=0}^L 
\frac{L!}{k! (L-k)!}
\braa{
\sum_{l=0}^L
\frac{(-1)^{L-l}l^k}{l! (L-l)!}
}
y^{L-k}.
\end{align*} 
Since $y$ is arbitrary, 
we have (\ref{magic}),
which completes the proof.

\end{proof}


\end{document}